\documentclass[12pt,a4paper, longbibliography]{article}
\usepackage[utf8]{inputenc}
\usepackage[T1]{fontenc}
\usepackage{authblk}

\usepackage{amsmath, amscd, amsthm, amscd, amsfonts, amssymb, graphicx, color, soul, enumerate, xcolor,  mathrsfs, latexsym, bigstrut, framed, caption}
\usepackage{pstricks,pst-node,pst-tree}
\usepackage[bookmarksnumbered, colorlinks, plainpages,backref]{hyperref}
\hypersetup{colorlinks=true,linkcolor=blue, anchorcolor=green, citecolor=midblue, urlcolor=darkblue, filecolor=magenta, pdftoolbar=true}
\usepackage[all]{xy}

\usepackage{tikz}
\usetikzlibrary{calc,decorations.pathreplacing,decorations.markings,positioning,shapes}

\newcommand{\Aut}{\ensuremath{\operatorname{Aut}}}

\newcommand{\id}{\ensuremath{\text{\rm id}}}

\newcommand\stackplus[1]{\makebox[0ex][l]{$+$} \raisebox{-.75ex}{\makebox[2ex]{$_{#1}$}}}

\definecolor{cupgreen}{rgb}{0,0.498,0.208}
\definecolor{cupblue}{rgb}{0,0,.5}
\definecolor{cupred}{rgb}{1,0.04,0}
\definecolor{cuppink}{rgb}{0.925,0,0.545}
\definecolor{cupmagenta}{rgb}{0.624,0.161,0.424}
\definecolor{cupbrown}{rgb}{0.71,0.212,0.133}

\definecolor{cupgreen}{rgb}{0,0,0}
\definecolor{cupblue}{rgb}{0,0,0}
\definecolor{cupred}{rgb}{0,0,0}
\definecolor{cuppink}{rgb}{0,0,0}
\definecolor{cupmagenta}{rgb}{0,0,0}
\definecolor{cupbrown}{rgb}{0,0,0}

\definecolor{TITLE}{rgb}{0,0,0}
\definecolor{midblue}{rgb}{0.00,0.0,0.80}
\definecolor{darkblue}{rgb}{0.00,0.00,0.45}
\definecolor{SECTION}{rgb}{0.50,0.00,1.00}
\definecolor{THM}{rgb}{0.8,0,0.1}
\definecolor{SEC}{rgb}{0,0,1}

\newcommand{\aut}{\mathrm{Aut}}

\newcommand{\stab}{\mathrm{Stab}}

\makeatother
\normalsize
\newtheorem{theorem}{{\color{THM} Theorem}}[section]

\DeclareRobustCommand{\stirling}{\genfrac\{\}{0pt}{}}

\newtheorem{lemma}[theorem]{{\color{THM}Lemma}}
\newtheorem{proposition}[theorem]{{\color{THM}Proposition}}
\newtheorem{corollary}[theorem]{{\color{THM}Corollary}}

\theoremstyle{definition}
\newtheorem{definition}[theorem]{{\color{THM}Definition\ }}

\numberwithin{equation}{section}

\date{}
\title{Distinguishing threshold for some graph operations}

\author[1]{\small M.H. Shekarriz\thanks{mhshekarriz@yazd.ac.ir}}

\affil[1]{Department of Mathematics, Yazd University, 89195-741, Yazd, Iran.}	

\author[2]{S.A. Talebpour\thanks{seyed.alireza.talebpour@gmail.com}}

\author[2]{B. Ahmadi\thanks{bahman.ahmadi@shirazu.ac.ir}}

\author[2]{M.H. Shirdareh Haghighi\thanks{shirdareh@shirazu.ac.ir}}

\author[1]{S. Alikhani\thanks{alikhani@yazd.ac.ir}}

\affil[2]{Department of Mathematics, Shiraz University, Shiraz, Iran}

\begin{document}

	\maketitle
	\begin{abstract}
		A vertex coloring of a graph $G$ is distinguishing if non-identity automorphisms do not preserve it. The distinguishing number, $D(G)$, is the minimum number of colors required for such a coloring and the distinguishing threshold, $\theta(G)$, is the minimum number of colors~$k$ such that any arbitrary $k$-coloring is distinguishing. Moreover, $\Phi_k (G)$ is the number of distinguishing coloring of $G$ using at most $k$ colors. In this paper, for some graph operations, namely, vertex-sum, rooted product, corona product and lexicographic product, we find formulae of the distinguishing number and threshold using $\Phi_k (G)$. 
		
		
	\end{abstract}

	\section{Introduction}\label{intro}

	A graph's symmetry is called an \emph{automorphism}. A vertex coloring of a graph is said to be \emph{distinguishing} (aka symmetry breaking) if the identity is the only automorphism that preserves it. The \emph{distinguishing number} of a graph $G$, denoted by $D(G)$, is the smallest number of colors required for such a coloring. For a positive integer $d$, if $G$ can be distinguishingly  colored by $d$ colors, we say that $G$ is $d$-distinguishable. 
	
	The concept is not new and dates back  to 1970s, when Babai considered it for the first time~\cite{Babai1977}. However, after appearance of~\cite{albertson1996symmetry} by Albertson and Collins in 1996, the concept captured a lot of interests and large number of novel methods and results were introduced to the community. Here we only mention those that are related to our results.
	
	
	Many results in the literature is about product graphs. For example, 2-distinguishability of the hypercubes $Q_k$, for $k\geq 4$, is shown by Bogstad and Cowen~\cite{Bogstad2004} while Imrich and Klav{\v{z}}ar in \cite{Imrich2006CartPower} generalized it by showing that the distinguishing number of any Cartesian power of a connected graph is~2 with some very few exceptions. Moreover, the Cartesian products of relatively prime graphs $G$ and $H$, whose sizes are close to each other, are shown to be distinguishable by a small number of colors by Imrich, Jerebic and Klav\v{z}ar \cite{Imrich2008CartComp},  while Estaji et al. in \cite{Estaji} demonstrate how close these sizes must be to have $D(G\square H)=2$.
	
	The  \emph{corona product} was also studied for their distinguishing indices by Alikhani and Soltani in~\cite{Alikhani2017corona}.
	We will remind their results in Section \ref{corona}. The \emph{lexicographic product}, is another interesting graph operation from symmetry breaking point of view. Alikhani and Soltani showed in another paper~\cite{Alikhani2018} that under some conditions on the automorphism group of a graph $G$, we have $D(G)\leq D(G^k)\leq D(G)+k-1$, where $G^k$ is the $k$'th lexicographic power of $G$. Moreover they showed that if $G$ and $H$ are connected, then  $D(H)\leq D(G\circ H)\leq |G|\cdot D(H)$, where $G\circ H$ stands for the lexicographic product of $G$ and $H$.
	
	In order to accurately calculate the distinguishing number of the lexicographic product, $D(G\circ H)$, Ahmadi, Alinaghipour and Shekarriz~\cite{ahmadi2020number} defined some indices such as $\Phi_k(G)$, which stands for the number of non-equivalent distinguishing colorings of the graph $G$ with at most $k$ colors. They proved that $D(G\circ H)=k$, where $k$ is the least integer such that $\Phi_k (H)\geq D(G)$, provided that  the automorphisms of $G\circ H$ are all natural~\cite{heminger1968}.
	
	To overcome some computation difficulties, the authors of~\cite{ahmadi2020number} further defined the \textit{distinguishing threshold} $\theta(G)$ as the minimum number $k$ of colors such that any coloring of the graph $G$ with $k$ colors is distinguishing. They showed that $\theta(K_n)=\theta(\overline{K_n})=n$, 
 $\theta(K_{m,n})=m+n$, $\theta(P_n)=\lceil\frac{n}{2}\rceil+1$, for $n\geq 2$, and $\theta(C_n)=\lfloor\frac{n}{2}\rfloor+2$, for $n\geq 3$. Moreover, for $n\geq 5$, they proved that $\theta(K(n,2))=\frac{1}{2}(n^2-3n+6)$, where $K(n,k)$ stands for the Kneser graph.  Calculating this index for several classes of product graphs is one of our intentions throughout this paper.
	
	The distinguishing threshold turns out  to be an interesting index on its own. Shekarriz~et~al.~\cite{ShekarrizAhmadiTH2021-theta} considered this index more thoroughly. Their study revealed that this index is related to the cycle structure of the automorphisms of the graph. They also showed that $\theta(G)=2$ if and only if $\vert G\vert = 2$. Moreover, they studied graphs whose distinguishing threshold is 3 and graphs with $D(G)=\theta(G)$. Furthermore, they considered the distinguishing  threshold for graphs in the Johnson scheme.
	
	The distinguishing threshold for the Cartesian product of connected graphs has recently been studied. Alikhani and Shekarriz~\cite{Shekarriz2021Cartesian} showed that when $G=G_1\square G_2 \square \ldots \square G_k$ is a prime factorization to mutually non-isomorphic connected graphs, we have 
	\[
	\theta (G)=\max \left\{ \left(\theta(G_i)-1\right)\cdot \frac{\vert G\vert}{\vert G_i \vert} \; : \; i=1,\dots,k \right\}+1.
	\]
	 Meanwhile, they showed that when $G$ is a connected prime graph and $k\geq 2$ is a positive integer, we have 
	\[
	\theta (G^k)=\vert G\vert^{k-1}\cdot\max\left\{ \frac{\vert G\vert +1}{2}, \left(\theta(G)-1\right)\right\}+1,
	\]
	where $G^k$ is the $k$th Cartesian power of $G$. 
	
	There are still some interesting graph operations whose symmetry breaking indices are yet to be calculated.	The \emph{vertex-sum}, whose definition is mentioned in Section \ref{vertex_sum}, is one of these operations. For some results on this graph operation, the reader may refer, for example, to~\cite{Barioli2004,Huang2010}. We found this operation interesting enough to be considered for their symmetry breaking indices.
	
	We start our study by first recalling  some fundamental results in Section~\ref{prelim}. Vertex-sum graphs are studied for breaking their symmetries in Section~\ref{vertex_sum}. The distinguishing number and the distinguishing threshold for the rooted products are  considered in Section~\ref{rooted} after considering their automorphism groups. For the corona products in Section~\ref{corona}, we not only refine Alikhani and Soltani's result in~\cite{Alikhani2017corona} by exactly calculating their distinguishing number, but also calculate their distinguishing threshold. Finally, the distinguishing threshold of the  lexicographic product graphs is evaluated  in Section~\ref{lexico}.
	
	All graphs in this paper are assumed to be  finite and simple, and the groups are finite. Undefined concepts and notation can be found in \cite{diestel2017} by Diestel. 
	
	\section{Preliminaries}\label{prelim}
	
	Two colorings $c_1$ and $c_2$ of a graph $G$ are said to be  \textit{equivalent} if there is an automorphism $\alpha\in\aut(G)$ such that $c_1(v) = c_2(\alpha(v))$, for all $v \in V(G)$. The  number of non-equivalent distinguishing colorings of $G$ with colors from the set $\{1, \ldots, k\}$ is denoted by $\Phi_k(G)$, while $\varphi_k(G)$ stands for the number of non-equivalent $k$-distinguishing colorings of $G$ (i.e., the number of non-equivalent colorings which use exactly $k$ colors)~\cite{ahmadi2020number}. The following formula reveals the relationship between these two indices:
	
	\[
	\Phi_k(G) = \sum_{i=D(G)}^k {k\choose i} \varphi_i(G).
	\]
	
	\noindent Using simple counting argument, it can easily be shown that for $ n,k \geq 1$ we have
	\begin{equation}\label{Phi_of_paths}
		\Phi_k(P_n) = \frac{1}{2}(k^n - k^{\lceil \frac{n}{2} \rceil}), 
	\end{equation}
	and $\Phi_k(K_n) = {k \choose n}$ when $n\geq 2$ and $k \geq n$~\cite{ahmadi2020number}. 
	
	Knowing the distinguishing threshold of a graph $G$ makes calculations of $\varphi_k(G)$ easier when $k$ is large enough: when every $k$-coloring of $G$ is distinguishing, i.e., when $k\geq\theta(G)$, we have 
	
	\begin{equation}
		\varphi_k (G)=\frac{k! \stirling{n}{k}}{|\aut(G)|},
	\end{equation} 
	
	\noindent where $\stirling{n}{k}$ denotes the Stirling number of the second kind~\cite{ahmadi2020number}.
	
	When $G$ is a disconnected graph, its distinguishing threshold is given by the following theorem. To state it, we will follow the notation from~\cite{ShekarrizAhmadiTH2021-theta}. Suppose that $G$ is a graph with connected components $G_1, \ldots, G_k$, where all   $G_i$ are asymmetric. Then, we consider the isomorphism congruence classes $\mathcal{C}_1,\ldots, \mathcal{C}_m$ of the graphs $G_1,\ldots, G_k$, where we assume that the $\mathcal{C}_j$'s are increasingly ordered in the sense that if $j<\ell$, and $G_{i_j}\in \mathcal{C}_j$ and $G_{i_\ell}\in \mathcal{C}_\ell$, then $|G_{i_j}|\leq |G_{i_\ell}|$. Define $\nu(G)$ to be $|G_{i_s}|$, where $G_{i_s}\in \mathcal{C}_s$ and~$s$ is the smallest integer with the property that $|\mathcal{C}_s|>1$; and if there is no such $s$, then define $\nu(G)$ to be $|G|$. Note that, for example, if $k=1$ (i.e., if $G$ is a connected asymmetric graph), then $\nu(G)=|G|$.

	\begin{theorem}\label{union}\textnormal{\cite{ShekarrizAhmadiTH2021-theta}}
		Let $G_1,   \ldots, G_k$ be arbitrary connected   graphs and let $G = \cup_{i=1}^{k} G_i$.  
		\begin{enumerate}[(a)]
			\item If $\aut(G_i) \ne \{\mathrm{id}\}$, for all $1 \leq i \leq n$, then 
			\[\theta(G) = \max_{1\leq i\leq k} \; \left\{\theta(G_i) + \sum_{j \neq i}|G_j| \right\}.\]
			
			\item  If $\aut(G_i) = \{\mathrm{id}\}$, for all $1 \leq i \leq k$, then $\theta(G)=|G|-\nu(G)+1$.
			
			\item If $\{1,\ldots,k\}=A\cup B$ is a non-trivial partition, $\aut(G_i) \ne  \{\mathrm{id}\}$, for $i\in A$, and $\aut(G_i) = \{\mathrm{id}\}$, for  $i \in B$, then set  $G_A = \bigcup_{i \in A} G_i$ and $G_B = \bigcup_{i \in B} G_i$. In this case, we have
			\[\theta(G) = \max \{ \theta(G_A) + |G_B|,\; \theta(G_B) + |G_A| \},\]
			unless   $G_B$ is asymmetric and $\theta(G_A) + |G_B|\leq  \theta(G_B) + |G_A|$, in which case  we have $\theta(G) = \theta(G_A) + |G_B|$.\qed
		\end{enumerate}
	\end{theorem}
	
	It was mentioned in the introduction that the distinguishing threshold of a graph $G$ is related to  the cycle structure of its automorphism group $\aut(G)$. Formally, when $\alpha\in\aut (G)$ and $v\in G$, the ordered tuple  $\sigma=(v, \alpha(v), \alpha^2 (v),\ldots,\alpha^{r-1} (v))$ forms a cycle of length $r$ provided that $r$ is the smallest integer such that  $\alpha^{r}(v)=v$. Note that a length-one cycle $(v)$ fixes the vertex $v$ and that the order $o(\sigma)$ of a cycle $\sigma$ in the automorphism group is the length of $\sigma$. Following the notation of~\cite{ShekarrizAhmadiTH2021-theta}, we denote the number of cycles of a nontrivial automorphism $\alpha$  by $c( \alpha)$ while $c(\id)=0$. For example if $\alpha=(1 3)(4 5)$ is an automorphism of a graph $G$ with $V(G)=\{1,\ldots,5\}$, then $c( \alpha)=3$.  Then, for any graph $G$ we have~\cite{ShekarrizAhmadiTH2021-theta}
	\begin{equation}\label{max-lem}
		\theta(G)=\max\left\{c(\alpha)\;:\; \alpha\in \aut(G)\right\}+1.
	\end{equation}

	\section{Vertex-sum}\label{vertex_sum}
	
	In this section, we focus on the distinguishing number and the distinguishing threshold of some vertex-sum graphs. Given  disjoint graphs $G_1,\ldots,G_k$ with $u_i\in V(G_i)$, $i=1,\ldots,k$, the vertex-sum of $G_1,\ldots,G_k$,  at the vertices $u_1,\ldots,u_k$, is the graph $G_1\stackplus{u} G_2 \stackplus{u} \cdots \stackplus{u}  G_k$ obtained from $G_1,\ldots,G_k$ by identifying the vertices $u_i$, $i=1,\ldots,k$, as the same vertex $u$. This definition is from~\cite{Barioli2004} by Barioli, Fallat and Hogben. We call $u$ the \textit{central vertex} of the vertex-sum.	The vertex-sum of $t$ copies of a graph $G$ at a vertex $u$ is denoted by $G_u^t$, $t \geq 2$.	For the sake of simplicity,	we may assume that the vertex $u$ belongs to all the  $G_i$.
	
	Recall that when a group $\Gamma$ acts on a set $X$ and $x\in X$, $\stab_\Gamma(x)$ denotes the stabilizer of $x$ in $\Gamma$, that is, the subgroup of $\Gamma$ consisting of the elements $g$ such that $g(x)=x$.
	Let $G$ be a graph and consider the natural action of $\aut(G)$ on $V(G)$ and  fix a  $u \in V(G)$. It is straightforward to see that the function which maps any $\alpha \in \stab_{\aut(G)}(u)$ to its restriction on $G-u$ is a group injection. This proves the following well-known fact.

	\begin{lemma}\label{stab_subgroup}
		Let $G$ be a graph and $u \in V(G)$. Then $\stab_{\aut(G)}(u)$ is isomorphic to a subgroup of $\aut(G-u)$. \qed
	\end{lemma}
	This  motivates us  to consider the vertices $u\in V(G)$ for which these two groups are isomorphic. 
	
	\begin{definition} Let $G$ be a graph and $u\in V(G)$. We say that $u$ is a  \textit{steady} vertex of $G$ if $\stab_{\aut(G)}(u)\cong \aut(G-u)$.
	\end{definition}
	Note that $u$ is a steady vertex of $G$ if and only if for any $\alpha \in \aut(G-u)$, $\alpha$ maps the neighborhood $N_G(u)$ of $u$ onto itself. 
	Therefore when $u$ is a steady vertex, $D(G-u)$ must be equal to either  $D(G)-1$ or $D(G)$, because in a distinguishing coloring of $G$, $u$ is either uniquely colored or not. As some easy examples, note that every vertex of a regular graph is steady, while no vertex of the path $P_n$ is steady, for $n\geq 4$.
		
	Later in this section, we will see that being steady is an important property for evaluating the distinguishing number of some vertex-sum graphs. However, before addressing this problem, we   take a closer look at this property. An interesting result is the relationship between the rather algebraic concept of being steady and the graph theoretical concept of distinguishing coloring. More specifically,  we show that for a connected graph $G$, except for some small families of graphs, an equivalent condition for a vertex $u\in V(G)$ to be steady is that every distinguishing coloring of $G$ induces a distinguishing coloring on $G-u$ (Theorem~\ref{steady_iff}). First, we state one direction of the claim in the following proposition whose proof is trivial.
	
	\begin{proposition}\label{steady_then_disting}
		Let $G$ be a connected graph. If a vertex $u\in V(G)$ is steady, then every distinguishing coloring of $G$ induces a distinguishing coloring on $G-u$.\qed
	\end{proposition}

The following lemma  will be used to prove the converse  direction of Proposition~\ref{steady_then_disting} for a large family of graphs.

	\begin{lemma}\label{partition_A_B}
		Suppose that $G$ is a graph, $p$ is a prime, $Q=\{0,\ldots,p-1\}\subseteq V(G)$ and the cycle $(0,1,\ldots,p-1)$ appears in the cycle decomposition of an automorphism $\alpha$ of $G$. Suppose also that $\beta\in \aut(G)$ fixes $Q$ set-wise and is not the identity on $Q$. If $Q$ is partitioned into two nonempty sets $A$ and $B$, then there exists $0<\ell<p$ such that the automorphism $\alpha^{-\ell} \beta \alpha^{\ell}$ maps a vertex of $A$ to a vertex of $B$.
	\end{lemma}
	\begin{proof}
		All differences and additions below take place in the group $\mathbb{Z}_p$. It suffices to consider the restrictions of $\alpha$ and $\beta$ to $Q$. 
		Let $\beta=\sigma_1\cdots\sigma_r$ be the cycle decomposition of $\beta$ where $\sigma_i=(t_1^i, t_2^i,\ldots, t_{o(\sigma_i)}^i)$, for $i=1,\ldots, r$.  Note that 
		\[\alpha^{-1}\beta \alpha=\sigma_1'\sigma_2'\cdots \sigma_r',\]
		where, $\sigma_i'=(t_1^i-1, t_2^i-1, \ldots, t_{o(\sigma_i)}^i-1)$.
	
Suppose on contrary that the claim does not hold. Moreover, assume that $0\in A$ and $(\alpha^{-1} \beta \alpha)(0)=x$. Thus, in particular we have  $x\in A$.   Note that for any positive integer $\ell$,
\begin{equation}\label{ell_conjugate}
(\alpha^{-(\ell+1)} \beta  \alpha^{(\ell+1)})(-\ell)=x-\ell.
\end{equation}
According to the contrary assumption, for any $\ell$,  either both vertices $-\ell$  and $x-\ell$ belong to $A$ or both of them belong to $B$.
The equation~(\ref{ell_conjugate}), further  implies that
\[ 
\alpha^{kx-1}\beta \alpha^{-(kx-1)} (kx) = (k+1)x, \quad k\geq 0.
 \]
 Therefore, all multiples of $x$ belong to $A$, and since $\gcd(p,x)=1$, this implies that $A=Q$, which is a contradiction.
\end{proof}

For the next theorem, we will also make use of the notion of a \textit{gear graph} which is defined as follows. A graph $G$ is said to be a $p$-gear graph if $V(G)$ can be partitioned into $m$ subsets $V_1, V_2, \ldots, V_m$ in which, for a prime number $p$, $|V_i| = p^{n_i}$ such that for some $1 \leq i \leq m$, $n_i \geq 2$ and, for all $1 \leq i \leq m$, the induced subgraph $G[V_i]$ is a circulant graph. 	
We are, now, ready to state the following theorem which displays the importance of steady vertices in distinguishing coloring of graphs.

\begin{theorem}\label{steady_iff}
Let $G$ be a connected graph. If $u \in V(G)$ is a steady vertex of $G$, then every distinguishing coloring of $G$ induces a distinguishing coloring on $G-u$. Furthermore, if any distinguishing coloring of $G$ induces a distinguishing coloring on $G-u$ and $G-u$ is not a $p$-gear graph, then $u$ is a steady vertex.
\end{theorem}
\begin{proof}
		According to Proposition~\ref{steady_then_disting}, it suffices to prove the second part. Assume that $G-u$ is not a $p$-gear graph and any distinguishing coloring of $G$ induces a distinguishing coloring on $G-u$. Suppose that $u$ is not steady. This means that  there exists an automorphism $\alpha \in \aut(G-u)$ which cannot be extended to any automorphism in $\stab_{\aut(G)}(u)$. Thus, there are two vertices $v,w\in G-u$ such that  $v\sim u$ but $w\nsim u$ in $G$ and  $\alpha(v)=w$. Let $\mathcal{A}$ be the set of all such automorphisms $\alpha$, and  let $\lambda\in \mathcal{A}$  be such that 
		\[
		c(\lambda)= \max\{ c(\alpha)\;:\; \alpha \in \mathcal{A} \}.
		\]
		Let $\lambda=\sigma_1 \sigma_2 \ldots \sigma_r$ be its cycle decomposition and $c(\lambda)=r+s$, where $s$ is the number of vertices that are fixed by~$\lambda$. Note that since $c(\lambda)$ is the maximum and $G$ is not a p-gear-graph, the orders $o(\sigma_j)$, $1\leq j\leq r$, are all equal to a prime number $p$. Let $\sigma_j=(x_{1}^j,\ldots, x_{p}^j)$ and $M_j=\{x_{1}^j,\ldots, x_{p}^j\}$ and consider a coloring $c$ of $G$ with $1+r+s$ colors   as follows: for $j=1,\ldots,r$, color all the vertices in $M_j$ by the color $j$ and  color any other vertices of $G$ (including $u$) by a new color.  It is clear that $\lambda$ preserves the restriction of $c$ to $G-u$ and hence $c$ does not induce a distinguishing coloring on $G-u$.  Therefore, in order to conclude the proof, it is enough to show that $c$ is a distinguishing coloring for $G$.
		
		Assume that $\beta\in \aut(G)$ is a non-identity automorphism that preserves $c$. We deduce that $\beta(M_j)=M_j$, for each $j=1,\ldots,r$, and $\beta$ fixes all other vertices in $G$. Since $\beta\in  \stab_{\aut(G)}(u)$, the restriction $\tilde{\beta}$ of $\beta$ to $G-u$ is an automorphism of $G-u$. For some $1\leq j \leq r$, the set $M_j$ contains a vertex $v$ such that $v\sim u$ while  $\lambda(v)\nsim u$. Without loss of generality, we assume that $v\in M_1$.
		
		\begin{quotation}
			\noindent\textbf{Claim.} {\it $\tilde{\beta}$ is not the  identity on $M_1$}.

			\noindent\textit{Proof.}  First consider the case in which there exists an edge between $M_i$ and $M_j$, for some $i,j\in \{1,\ldots,r\}$, where $\tilde{\beta}$ is identity on $M_i$ but non-identity on $M_j$. Let $w=x_{k}^j$ and $z=x_{\ell}^j$ be distinct vertices of $M_j$ such that $\tilde{\beta}(w) = z$. It follows that the set $N$ of neighbors of $w$ in $M_i$ is the same as the set of neighbors of $z$ in $M_i$. Let $k - \ell = t \;(\mathrm{mod}\; p)$. If $a \in N$, then  $\lambda^t(a + ht) \in N$, for all $1 \leq h \leq p$.  Since $p$ is prime, we deduce that $N = M_i$; that is, each vertex in  $M_i$ is adjacent to each vertex in $M_j$.
			
			Set $T=\{j \;:\;  \tilde{\beta} \; \;\text{is identity on}~ M_j\}$. We have $1\in T$ and $T'=\{1,\ldots,r\} -  T \neq \emptyset$.
			Define $ \mu =\prod_{i\in T } \sigma_i.$ We show that $\mu$ is an automorphism of $G-u.$ Let $x$ and $y$ be two arbitrary adjacent vertices in $G-u$. We show  that  $\mu (x)$ is adjacent to $\mu (y)$, as well. If both $x$ and $y$ are fixed by $\mu$, then there is nothing to prove.
			We assume $x\in M_i$, for some $i\in T$, so $\mu(x)=\sigma_i(x)=\lambda(x)$. If $y$ also belongs to $M_i$, then $\mu(y)=\sigma_i(y)=\lambda(y)$ and, since $\lambda$ is an automorphism, we have $\mu(x)\sim \mu(y)$. Thus we assume  $y\in M_j$, for some $j\neq i$.  If $j\in T$, then $\mu(y)=\sigma_j(y)=\lambda(y) \sim \lambda(x)=\mu(x).$  If $j \notin T$, then $\mu(y)=y$ and, as we showed above, each vertex of $M_i$ is adjacent to each vertex of $M_j$.  In particular, $\mu(x)\sim \mu(y)$. Finally,
			if $y\notin M_i$, for all $1\leq i\leq r$, then $\mu(y)=y=\lambda(y)$ and again $\mu(x)\sim \mu(y)$.
			
			Since $\mu=\lambda$ on $M_1$, it follows that $\mu$ is also in $\mathcal{A}$, while 
			\[
			c(\mu)=|T|+\sum_{i\in T'}|M_i|+s>r+s=c(\lambda).
			\]
			This is a contradiction, which implies that $\tilde{\beta}$ is not the identity on $M_1$ and nor is $\beta$, which proves the claim.   
		\end{quotation}

		To complete the proof of the theorem, we partition $M_1$ into nonempty sets $A$ and $B$ where $A$ is the set of the vertices adjacent to $u$ in $G$.    
		According to Lemma~\ref{partition_A_B}, there exists an  $0<\ell<p$ such that the automorphism $\pi=\lambda^{-\ell} \, \tilde{\beta} \,\lambda^{\ell}$ of $G-u$ maps a vertex of $A$ to a vertex of $B$. This shows that $\pi\in \mathcal{A}.$ Also, since $\pi$ and $\tilde{\beta}$ are conjugate, $c(\pi)=c(\tilde{\beta})$. On the other hand,  $\tilde{\beta}(A)=A$ and $\tilde{\beta}$ has at least two cycles on  $M_1$.  Thus $c(\pi)>c(\lambda)$, which contradicts the definition of $\lambda$. This contradiction shows that $\tilde{\beta}$ is the identity on $M_1$ which, in turn, contradicts the claim and, therefore, the proof is done.
	\end{proof}

	Next, as   a straightforward use of steady vertices, we evaluate the distinguishing number of the vertex-sum of $t$ copies of a single graph at a given vertex. We will make use of the   following lemma which is easy to prove.
	
	\begin{lemma}\label{steady_in_whole_vsum}
		Let $G_1, \ldots, G_t$ be connected graphs. A vertex $u$ is   steady   in each $G_i$ if and only if $u$ is steady in the vertex-sum of $G_1,\ldots,G_t$ at $u$.\qed
	\end{lemma}
	Recall that the \emph{eccentricity} of a vertex $v$ in  a graph $G$ is the maximum distance of $u$ from other vertices of $G$.
	
	\begin{theorem}\label{dvsum}
		Given a connected graph $G$ and a vertex $u\in V(G)$, we have
		\[D(G_u^t)\leq \min \left\{ k \,  : \, \Phi_k(G-u) \geq t \right\}. \] 
		Moreover, equality holds when $u$ is a steady vertex of $G$.
	\end{theorem}

	\begin{proof}
		First note that  every automorphism of $G_u^t$ fixes $u$, because $u$ is the unique vertex of $G_u^t$ with the minimum eccentricity. 
		Let $q=\min \left\{ k \,  : \, \Phi_k(G-u) \geq t \right\}$ and $\mathcal{A}_q = \{ c_1 , ... , c_{\Phi_{q}(G-u)}\} $ be the set of non-equivalent distinguishing colorings of $G-u$ with at most $q$ colors. Consider the coloring of  $G_u^t$ in which   each of the $t$ different copies of $G-u$ in $G_u^t$ is colored according to a distinct coloring in $\mathcal{A}_q$. Since the central vertex $u$ is fixed by any automorphism of $G_u^t$ and can be given any color, this coloring is clearly distinguishing. Thus the first part of the theorem follows. 
		
		Now  we assume that $u$ is a steady vertex of $G$. According to Lemma~\ref{steady_in_whole_vsum} and Theorem~\ref{steady_iff}, every distinguishing coloring of  $G_u^t$ must induce a distinguishing coloring  on  $G_u^t-u$.

		Now suppose    $ \Phi_k(G-u) < t $ and that  we want  to color $G_u^t$ using  $k$ colors. Then, even if the copies of $G-u$ in $G_u^t-u$ are  colored distinguishingly, there are two of them, say $G_1$ and $G_2$, which  have equivalent distinguishing colorings. Thus the automorphism of $G_u^t-u$ which only swaps these two copies preserves this coloring. We deduce  that this coloring is not distinguishing for $G_u^t-u$ and, consequently, $D(G_u^t)>k$ and the proof is complete.
	\end{proof}
	
	As an immediate consequence of Theorems~\ref{dvsum}, one observes the following.
	
	\begin{corollary}\label{vsum_of_K_n}
		If $G$ is the vertex-sum of  $t$ copies of $K_n$ at any vertex, then
		\[
		D(G) = \min \left\{    k  \,  :  \, {k \choose n-1}\geq t  \right\}.
	 \]
		In particular, for the cases $n=3,4,5$, we have
		\begin{align*}
		D\left( {(K_3)}_u^t\right) &= \left\lceil{\frac{1+\sqrt{8t+1}}{2}}\right\rceil,\\[10pt]
		D\left( {(K_4)}_u^t\right) &= \left\lceil \frac{\sqrt[3]{81t+3\sqrt{729t^2-3}}}{3}    +   \frac{1}{\sqrt[3]{81t+3\sqrt{729t^2-3}}}\right\rceil,\\[10pt]
        D\left( {(K_5)}_u^t\right) &= \left\lceil    \frac{3   + \sqrt{5+4\sqrt{24t+1}}}{2}\right\rceil.\qed
	\end{align*}
	\end{corollary}

	We note that if  $u$ is not steady, either strict inequality or equality may hold in Theorem~\ref{dvsum}. For example, let $H$ be the graph obtained by removing an edge from $K_4$ and $u$ be a vertex of degree $2$ in $H$. Then $u$ is not steady in $H$,  $\Phi_k(H-u)={k\choose 3}$ 	and $\min\left\{k \,\middle\vert\, {k\choose 3} \geq 2\right\} = 3$; however $D(H_u^2)=2$. 
	As another example, since $\Phi_1(P_3)=0 $ and $\Phi_2(P_3)=2$, the equality in Theorem~\ref{dvsum} holds for the vertex-sum of $P_4$ with itself at one of the end vertices, while none of the  vertices of $P_4$ is steady.

	The case of vertex-sum of triangles has already been stated in~\cite{alikhani2016distinguishing}. For  the vertex-sums of cycles of the same length, using the formula~\ref{Phi_of_paths}, we obtain the following result.
	
	\begin{corollary}\label{vsum_cycles_same_length}
		Let $G$ be the vertex-sum of $t$ cycles of length $n$. Then
		\[D(G) = \min \left\{ k  \; : \;   k^{n-1} - k^{\lceil \frac{n-1}{2} \rceil} \geq 2t  \right\}.\] 
		In particular, for the case where $n$ is  odd, we have 	
		\[\pushQED{\qed}
		D(G) =\left\lceil{\sqrt[\frac{n-1}{2}]{\frac{1 + \sqrt{8t+1}}{2}}}\,\,\right\rceil.\qedhere\popQED\] 
	\end{corollary}

	Now, turn our attention to the  vertex-sums of non-isomorphic graphs. It turns out that the evaluation for the case of  $2$-connected non-isomorphic graphs is straightforward.

	\begin{theorem}\label{gdvsum}
		Let $G_1, \ldots, G_t$ be 2-connected mutually non-isomorphic graphs all having $u$ as a steady vertex, and let $G$ be their vertex-sum at  $u$.
		Then $D(G)=\max_i \{ D(G_i-u) \} $.
	\end{theorem}
	\begin{proof}
		The vertex $u$ is fixed by any automorphism of $G$ because it is the only cut vertex of~$G$. The automorphisms of $G$ map  each $G_i$ on itself. Since $u$ is steady in all $G_i$, one can distinguishingly color $G$ with $\max_i \{D(G_i-u)\}$ colors. On the other hand, no distinguishing coloring of $G$ can be found using less than  $\max_i \{D(G_i-u)\}$ colors. Hence the result follows.
	\end{proof}
	
	In order to explain the necessity  of the assumptions of Theorem~\ref{gdvsum}, we note that in Figure~\ref{fig:example}, both graphs $G_1$ and $G_2$ are 2-connected,  the vertex  $u$ is not steady in $G_2$, $D(G_1-u)=2$, $D(G_2-u)=3$ and $D(G_1\stackplus{u} G_2)=2$.
	Furthermore, in Figure~\ref{fig:example2}, the vertex  $u$ is steady in both $H_1$ and $H_2$, $D(H_1-u)=3$, $D(H_2-u)=2$  and $D(H_1\stackplus{u} H_2)=5$.
	
	\begin{figure}[h]
		\centering
		\begin{center}
			\begin{tabular}{ c c c c  c c c c }
				\begin{tikzpicture}  [scale=0.6]
					\tikzstyle{every path}=[line width=1pt]
					
					\newdimen\ms
					\ms=0.1cm
					\tikzstyle{s1}=[color=black,fill,rectangle,inner sep=3]
					\tikzstyle{c1}=[color=black,fill,circle,inner sep={\ms/8},minimum size=2*\ms]
					
					
					\coordinate (a1) at  (0,2);
					\coordinate (a3) at (2,1);
					\coordinate (a5) at (0,0);

					\draw [color=black] (a1) -- (a3);
					\draw [color=black] (a3) -- (a5);
					\draw [color=black] (a5) -- (a1);
					
					\draw (a1) coordinate[c1];
					\draw (a3) coordinate[c1,label=above:$u\textnormal{ }$];
					\draw (a5) coordinate[c1];
					
				\end{tikzpicture}
				&
				\begin{tikzpicture}  [scale=0.6]
					
					\tikzstyle{every path}=[line width=1pt]
					
					\newdimen\ms
					\ms=0.1cm
					\tikzstyle{s1}=[color=black,fill,rectangle,inner sep=3]
					\tikzstyle{c1}=[color=black,fill,circle,inner sep={\ms/8},minimum size=2*\ms]
					

					\coordinate (a6) at (4,1);
					\coordinate(a7) at (6, 2.5);
					\coordinate(a8) at (6 , -0.5);
					\coordinate(a9) at (8, 1);
					

					\draw [color=black] (a7) -- (a8);
					\draw [color=black] (a7) -- (a6);
					\draw [color=black] (a6) -- (a8);
					\draw [color=black] (a7) -- (a9);
					\draw [color=black] (a9) -- (a8);

					\draw (a6) coordinate[c1,label=above:$u\textnormal{ }$];
					\draw(a7) coordinate[c1];
					\draw(a8) coordinate[c1];
					\draw(a9) coordinate[c1];
				\end{tikzpicture}
				& & & &&
				\begin{tikzpicture}  [scale=0.6]
					\tikzstyle{every path}=[line width=1pt]
					
					\newdimen\ms
					\ms=0.1cm
					\tikzstyle{s1}=[color=black,fill,rectangle,inner sep=3]
					\tikzstyle{c1}=[color=black,fill,circle,inner sep={\ms/8},minimum size=2*\ms]
					

					\coordinate (a1) at  (0,2);
					\coordinate (a3) at (2,1);
					\coordinate (a5) at (0,0);
					\coordinate (a6) at (4,2.5);
					\coordinate (a7) at (4,-0.5);
					\coordinate (a8) at (6,1);
					
					
					\draw [color=black] (a1) -- (a3);
					\draw [color=black] (a3) -- (a5);
					\draw [color=black] (a5) -- (a1);
					\draw [color=black] (a3) -- (a6);
					\draw [color=black] (a7) -- (a6);
					\draw [color=black] (a7) -- (a3);
					\draw [color=black] (a6) -- (a8);
					\draw [color=black] (a7) -- (a8);
					
					\draw (a1) coordinate[c1];
					\draw (a3) coordinate[c1,label=above:$u\textnormal{ }$];
					\draw (a5) coordinate[c1];
					\draw (a6) coordinate[c1];
					\draw (a7) coordinate[c1];
					\draw (a8) coordinate[c1];
				\end{tikzpicture}
				\\ \vspace{1mm} \\
				$G_1$ &$G_2$& &  &&& $G_1 \stackplus{u}G_2$\\
			\end{tabular}

		\end{center}
		\caption{a vertex-sum of two 2-connected graphs in which $u$ is not steady in one of them}
		\label{fig:example}
	\end{figure}
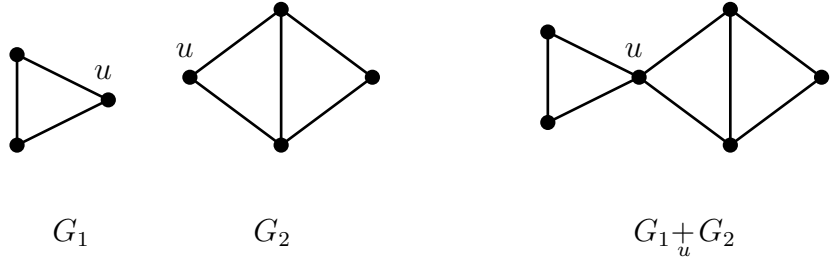

	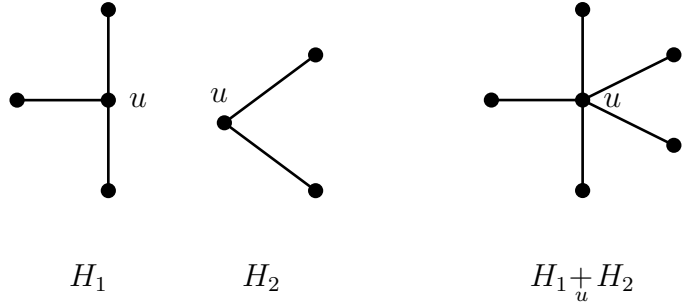
\begin{figure}[h]
		\centering
		\begin{center}
			\begin{tabular}{ c c c c ccc c }
				\begin{tikzpicture}  [scale=0.6]
					
					\tikzstyle{every path}=[line width=1pt]
					
					\newdimen\ms
					\ms=0.1cm
					\tikzstyle{s1}=[color=black,fill,rectangle,inner sep=3]
					\tikzstyle{c1}=[color=black,fill,circle,inner sep={\ms/8},minimum size=2*\ms]
					

					\coordinate (a1) at  (2,3);
					\coordinate (a3) at (2,1);
					\coordinate (a5) at (0,1);
					\coordinate (a6) at (2, -1);

					
					\draw [color=black] (a1) -- (a3);
					\draw [color=black] (a3) -- (a5);
					\draw[color=black](a6) -- (a3);
					
					
					\draw (a1) coordinate[c1];
					\draw (a3) coordinate[c1,label=right:$u\textnormal{ }$];
					\draw (a5) coordinate[c1];
					\draw(a6) coordinate[c1];
					
				\end{tikzpicture}
				&
				\begin{tikzpicture}  [scale=0.6]
					
					\tikzstyle{every path}=[line width=1pt]
					
					\newdimen\ms
					\ms=0.1cm
					\tikzstyle{s1}=[color=black,fill,rectangle,inner sep=3]
					\tikzstyle{c1}=[color=black,fill,circle,inner sep={\ms/8},minimum size=2*\ms]
					

					\coordinate (a1) at  (2,2.5);
					\coordinate (a3) at (0,1);
					\coordinate (a5) at (2,-0.5);

					
					\draw [color=black] (a1) -- (a3);
					\draw [color=black] (a3) -- (a5);

					
					\draw (a1) coordinate[c1];
					\draw (a3) coordinate[c1,label=above:$u\textnormal{ }$];
					\draw (a5) coordinate[c1];
					
				\end{tikzpicture}
				& &&&&
				\begin{tikzpicture}  [scale=0.4]
					\tikzstyle{every path}=[line width=1pt]
					
					\newdimen\ms
					\ms=0.1cm
					\tikzstyle{s1}=[color=black,fill,rectangle,inner sep=3]
					\tikzstyle{c1}=[color=black,fill,circle,inner sep={\ms/8},minimum size=2*\ms]
					

					\coordinate (a1) at  (0,0);
					\coordinate(a4) at (3,1.5);
					\coordinate(a6) at (3,-1.5);
					\coordinate(a3) at (-3,0);
					\coordinate(a5) at (0,3);
					\coordinate(a7) at (0,-3);


					\draw [color=black] (a1) -- (a3);
					\draw [color=black] (a1) -- (a4);
					\draw [color=black] (a1) -- (a5);
					\draw [color=black] (a1) -- (a6);
					\draw [color=black] (a1) -- (a7);
					
					\draw (a1) coordinate[c1,label=right:$u\textnormal{ }$];
					\draw (a3) coordinate[c1];
					\draw (a4) coordinate[c1];
					\draw (a5) coordinate[c1];
					\draw (a6) coordinate[c1];
					\draw (a7) coordinate[c1];
				\end{tikzpicture}
				\\ \vspace{1mm} \\
				$H_1$&$H_2$  &&&&&$H_1\stackplus{u} H_2$\\
			\end{tabular}

		\end{center}
		\caption{a vertex-sum of two graphs which are not 2-connected, albeit $u$ is steady in both of them}
		\label{fig:example2}
	\end{figure}

	In the rest of the section, we study the distinguishing threshold of vertex-sum graphs.  As an example, one can easily observe that 
	if $G$ is a vertex-sum of some complete graphs then $\theta(G) = |G|$. To deal with some non-trivial cases, we start by   the following theorem. 
	\begin{theorem}\label{theta-ord-2conn}
		Let $G$ be the vertex-sum at $u$ of the 2-connected graphs $G_1, \cdots , G_t$, which all have  $u$ as a steady vertex. If $G'$ is the disjoint union of $G_i-\{u\}$, $1\leq i\leq t$, then
		\[ \theta(G) = \theta(G')+1.\]
	\end{theorem}

	\begin{proof}
		Since $u$ is the only cut vertex of $G$, $\aut(G) = \stab_{\aut(G)}(u)$. In addition, according to Lemma~\ref{steady_in_whole_vsum},  $u$ is also steady in $G$. Hence, $$\aut(G) =  \stab_{\aut(G)}(u)\cong \aut(G-u)=\aut(G').$$ 
		Therefore, every coloring of $G$ with  $\theta(G')+1$ colors is distinguishing and hence $\theta(G)\leq \theta(G')+1.$
		On the other hand, consider a non-distinguishing coloring of $G'$ with 
		$\theta(G')-1$ colors and extend it to a coloring of $G$ by assigning  a new color to $u$. This coloring is not distinguishing for $G$. Thus $ \theta(G) = \theta(G')+1$.
	\end{proof}
	
	As a non-trivial example, we consider the distinguishing threshold of vertex-sum of cycles of the same length. The reader should note that this result can be proved directly, but using Theorems~\ref{union} and~\ref{theta-ord-2conn}, the proof is much shorter.
	
	\begin{theorem}
		Let $G$ be a vertex-sum of $t$ cycles of length $n$, Then
		\[\theta(G) = \left\lceil \frac{n-1}{2}\right\rceil + (n-1)(t-1)+2.\]
	\end{theorem}
	\begin{proof}
		Any cycle $C_n$ is 2-connected and, obviously, every vertex of it is steady. Therefore by Theorem \ref{theta-ord-2conn}, we have $\theta(G)=\theta (G')+1$ where $G'$ is the disjoint union of $t$ copies of $P_{n-1}$. The distinguishing threshold of this graph can be calculated using Theorem \ref{union} (a) as  
		\[\theta(G') = \max_{1\leq i\leq t} \; \left\{\theta(P_{n-1}) + \sum_{j \neq i}|P_{n-1}| \right\}=\left\lceil \frac{n-1}{2}\right\rceil+1 + (n-1)(t-1).\]
		Therefore, the result follows.
	\end{proof}

	\section{Rooted product}\label{rooted}
	A rooted graph $(G,v)$ is a graph $G$ with a vertex $v$ which is called \emph{root}. Two rooted graphs $(G_1 , v_1)$ and $(G_2 , v_2)$ are said to be isomorphic if there is an isomorphism $\alpha: G\rightarrow H$ such that $\alpha(v_1)=v_2$. By $\Aut(H,v)$ we mean all the automorphisms of $H$ that fix $v$; in other words, $\Aut(H,v)=\mathrm{Stab}_{\Aut(H)} (v)$. Moreover, $D(H,v)$, $\varphi_k (H,v)$, $\Phi_k (H,v)$ and $\theta (H,v)$ are defined as analogous symmetry breaking indices of $(H,v)$ by breaking $\Aut(H,v)$ instead of $\Aut(H)$. 
	
	The \emph{rooted product} was defined by Godsil and McKay in~\cite{godsil_mckay_1978} as follows. Let $G$ be a labeled graph on $n$	vertices and $\mathcal{L}$ be a sequence of $n$ rooted graphs $(H_1,v_1) ,\ldots, (H_n,v_n)$. Then the rooted product graph, $G(\mathcal{L})$, is  obtained by identifying $v_i$ with the $i$'th vertex of $G$.
	
	When $H_i$'s are not  isomorphic to each other, the automorphisms of the rooted product  can be very complicated because there might be some automorphisms that cannot be expressed  by the product factors. For example, suppose that $G=P_2$ and $\mathcal{L}=\left( (P_2,1) , (P_3,1)\right)$. Then $G(\mathcal{L})$ is $P_5$. While $G$ has no non-identity automorphism that fix a vertex, $G(\mathcal{L})$ has such an automorphism.
	
	However, when  $(H_1,v_1),\dots,(H_n,v_n)$ are all isomorphic to a rooted graph $(H,v)$ and both $G$ and $H$ are connected, all the automorphisms of the rooted product keep the product structure and they can be treated easier. This fact is shown in Theorem~\ref{aut-rooted}. For convenience, we limit our attention to this special case which is defined below and illustrated in Figure~\ref{fig:rooted}.
	
	\begin{definition}
		Let $G$ and $H$ be two graphs, $\vert G\vert=n$ and $v\in H$. Then the \emph{smooth rooted product}, denoted by $G_s(H,v)$, is the rooted product graph $G(\mathcal{L})$ such that $\mathcal{L}$ is a sequence of $n$ copies of $(H,v)$.
	\end{definition}
	
	\begin{figure}[h]
		\centering
		\begin{center}
			\begin{tabular}{ c c c c  c c c c }
				\begin{tikzpicture}  [scale=0.6]
					\tikzstyle{every path}=[line width=1pt]
					
					\newdimen\ms
					\ms=0.1cm
					\tikzstyle{s1}=[color=black,fill,rectangle,inner sep=3]
					\tikzstyle{c1}=[color=black,fill,circle,inner sep={\ms/8},minimum size=2*\ms]
					
					
					\coordinate (a1) at  (0,0);
					\coordinate (a2) at  (2,0);
					\coordinate (a3) at (4,1);
					\coordinate (a4) at (6,0);

					\draw [color=black] (a1) -- (a2);
					\draw [color=black] (a2) -- (a3);
					\draw [color=black] (a2) -- (a4);
					\draw [color=black] (a3) -- (a4);
					\draw (a1) coordinate[c1];
					\draw (a2) coordinate[c1];
					\draw (a3) coordinate[c1];
					\draw (a4) coordinate[c1];
					
				\end{tikzpicture}
				& &
				\begin{tikzpicture}  [scale=0.4]
					
					\tikzstyle{every path}=[line width=1pt]
					
					\newdimen\ms
					\ms=0.1cm
					\tikzstyle{s1}=[color=black,fill,rectangle,inner sep=3]
					\tikzstyle{c1}=[color=black,fill,circle,inner sep={\ms/8},minimum size=2*\ms]
					

					\coordinate (a6) at (0,0);
					\coordinate(a7) at (1,2);
					\coordinate(a8) at (2,0);
					\coordinate(a9) at (1,-2);
					

					\draw [color=black] (a6) -- (a7);
					\draw [color=black] (a7) -- (a8);
					\draw [color=black] (a8) -- (a9);
					\draw [color=black] (a6) -- (a9);

					\draw (a6) coordinate[c1];
					\draw(a7) coordinate[c1];
					\draw(a8) coordinate[c1];
					\draw(a9) coordinate[c1,label=below:$v$];
				\end{tikzpicture}
				& & &&
				\begin{tikzpicture}  [scale=0.8]
					\tikzstyle{every path}=[line width=1pt]
					
					\newdimen\ms
					\ms=0.1cm
					\tikzstyle{s1}=[color=black,fill,rectangle,inner sep=3]
					\tikzstyle{c1}=[color=black,fill,circle,inner sep={\ms/8},minimum size=1.6*\ms]
					
					
					\coordinate (a1) at  (0,0);
					\coordinate (a2) at  (2,0);
					\coordinate (a3) at (4,1);
					\coordinate (a4) at (6,0);
					\coordinate (a5) at  (0.5,1);
					\coordinate (a6) at  (0,2);
					\coordinate (a7) at (-0.5,1);
					\coordinate (a8) at  (2.5,1);
					\coordinate (a9) at  (2,2);
					\coordinate (a10) at (1.5,1);
					\coordinate (a11) at  (4.5,2);
					\coordinate (a12) at  (4,3);
					\coordinate (a13) at (3.5,2);
					\coordinate (a14) at  (6.5,1);
					\coordinate (a15) at  (6,2);
					\coordinate (a16) at (5.5,1);

					\draw [color=black] (a1) -- (a2);
					\draw [color=black] (a2) -- (a3);
					\draw [color=black] (a2) -- (a4);
					\draw [color=black] (a3) -- (a4);
					\draw [color=black] (a1) -- (a5);
					\draw [color=black] (a5) -- (a6);
					\draw [color=black] (a6) -- (a7);
					\draw [color=black] (a1) -- (a7);
					\draw [color=black] (a2) -- (a8);
					\draw [color=black] (a8) -- (a9);
					\draw [color=black] (a9) -- (a10);
					\draw [color=black] (a2) -- (a10);
					\draw [color=black] (a3) -- (a11);
					\draw [color=black] (a11) -- (a12);
					\draw [color=black] (a12) -- (a13);
					\draw [color=black] (a3) -- (a13);
					\draw [color=black] (a4) -- (a14);
					\draw [color=black] (a14) -- (a15);
					\draw [color=black] (a15) -- (a16);
					\draw [color=black] (a4) -- (a16);
					
					\draw (a1) coordinate[c1];
					\draw (a2) coordinate[c1];
					\draw (a3) coordinate[c1];
					\draw (a4) coordinate[c1];
					\draw (a5) coordinate[c1];
					\draw (a6) coordinate[c1];
					\draw (a7) coordinate[c1];
					\draw (a8) coordinate[c1];
					\draw (a9) coordinate[c1];
					\draw (a10) coordinate[c1];
					\draw (a11) coordinate[c1];
					\draw (a12) coordinate[c1];
					\draw (a13) coordinate[c1];
					\draw (a14) coordinate[c1];
					\draw (a15) coordinate[c1];
					\draw (a16) coordinate[c1];
				\end{tikzpicture}
				\\ \vspace{1mm} \\
				$G$ & &$H$& &  && $G_s (H)$\\
			\end{tabular}

		\end{center}
		\caption{the smooth rooted product of two graphs}
		\label{fig:rooted}
	\end{figure}
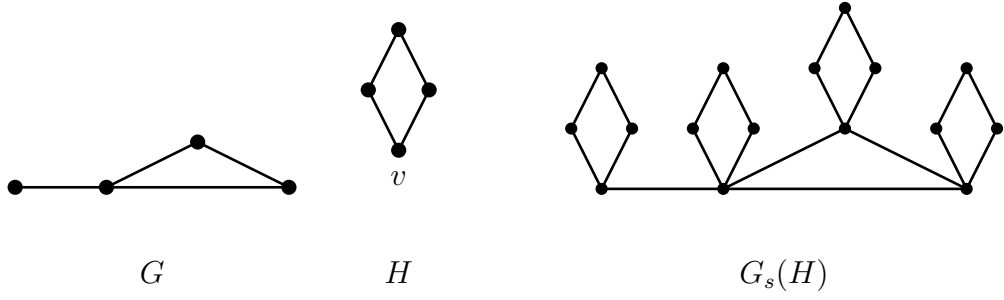

	It is easy to see that the smooth rooted product of $G_s(H,v)$ is a graph on $\vert G\vert \cdot \vert H \vert $ vertices and a subgraph of the Cartesian product $G\square H$, but it is not commutative. For this product, unlike the Cartesian product, every copy of $(H,v)$ can be automorphically mapped  onto itself separated from other such copies. We show such automorphisms by $n$-tuples $\left( \gamma_1,\ldots,\gamma_n\right)$, where $\gamma_i$ is an automorphism of the $i$'th copy of $(H,v)$. Consequently,  an arbitrary automorphism of  the smooth rooted product is of a simple form which is characterized  in the following theorem.
	
	\begin{theorem}\label{aut-rooted}
		Let $G$ and $H$ be connected graphs, $H$ be rooted at $v$ and $\vert G \vert =n\geq 2$. Then $\alpha\in \Aut(G_s(H,v))$ if and only if there are   $\beta\in\Aut(G)$ and $\gamma_1,\ldots , \gamma_n \in \Aut(H,v)$ such that 
		\[
		\alpha=\beta \circ \left( \gamma_1,\ldots,\gamma_n\right).
		\]
	\end{theorem}
	\begin{proof}
		It is obvious that $\beta \circ \left( \gamma_1,\ldots,\gamma_n\right)$ is an automorphism of $G_s(H,v)$. We prove that there is not any other. Suppose that $\alpha\in\Aut(G_s(H,v))$ and $V(G)=\{v_1,\ldots,v_n\}$. Then, one of the following two cases holds.
		
		\begin{itemize}
			\item[Case 1.] The set $V(G)$ is mapped by $\alpha$ onto itself. Then, $\alpha$ induces an automorphism on the subgraph induced by $V(G)$, which is isomorphic to $G$. Since $H$ is   connected, we deduce that $\alpha (H_i,v_i)=(H_j,v_j)$ when $\alpha(v_i)=v_j$. This means that $\alpha=\beta \circ \left( \gamma_1,\ldots,\gamma_n\right)$ for some $\beta\in\Aut(G)$ and $\gamma_1,\ldots , \gamma_n \in \Aut(H,v)$.
			\item[Case 2.] There is a vertex  $v_i$ such that $\alpha(v_i)\notin V(G)$; hence $\alpha(v_i)\in H_j$, for some $j$. Since $v_i$ is a cut vertex in $G_s(H,v)$, $\alpha (v_i)$ is also a cut vertex. Assume  that $\mathcal{A}_j$ is  the union of all connected components of $H_j  -  \{\alpha(v_i) \}$  which do  not contain $v_j$. We have 
			\[
			G_s(H,v) -  \{v_i\}=\left( H_i -  \{v_i\}\right) \cup \left(G -  \{v_i\}\right)_s(H)
			\]
			and
			\[
			G_s(H,v) -  \{\alpha(v_i)\}= \mathcal{A}_j\cup \left( G_s(H,v) -  \left(\mathcal{A}_j \cup \{\alpha(v_i)\}\right) \right) . 
			\]
			However, $G_s(H,v) -  \mathcal{A}_j\cup \{\alpha(v_i)\}$ is a connected component which cannot be embedded into\linebreak $\left(G -  \{v_i\}\right)_s(H)$ because the former  has more vertices. On the other hand, since $n\geq 2$, we have 
			\[
			\vert G_s(H,v) -  \mathcal{A}_j\vert\geq\vert H\vert>\vert H_i -  \{v_i\} \vert,
			\]
			which means that $G_s(H,v) -  \left(\mathcal{A}_j\cup \{\alpha(v_i)\}\right)$ cannot be mapped into $H_i -  \{v_i\}$ either. Consequently neither $H_i -  \{v_i\}$ nor $ \left(G -  \{v_i\}\right)_s(H)$ contains the connected component $G_s(H,v) -  \left(\mathcal{A}_j\cup \{\alpha(v_i)\}\right)$, which is a contradiction.
		\end{itemize}
		Therefore, the result follows.
	\end{proof}
	
	Using Theorem~\ref{aut-rooted}, we can evaluate the distinguishing number and the distinguishing threshold of the smooth rooted products.  The results appear in the next two theorems.

	 	\begin{theorem}\label{dist-rooted}
		Let $G$ and $H$ be two connected finite graphs, $H$ be rooted at $v$ and $\vert G\vert ,\vert H\vert\geq 2$. Then 
		\[
		D(G_s( H)) = \min\left\{k   \; :  \; \Phi_k (H,v)\geq\frac{D(G)}{k}\right\}.
		\]
	\end{theorem}
	\begin{proof}
		When $D(G)\leq D(H,v)$ we have $D(G_s(H,v))=D(H,v)$, because using this number of colors we can break all the automorphisms of $\Aut(H,v)$ and $\Aut(G)$ and hence, according to Theorem~\ref{aut-rooted}, all the automorphisms of $G_s(H,v)$ are also broken. In this case, the minimum $k$ such that  $k\cdot \Phi_k (H,v)\geq D(G)$ is $D(H,v)$ because $\Phi_k (H,v)=0$ for $k<D(H,v)$. Consequently, the statement follows.
		
		Now, suppose that $D(G)>D(H,v)$, $V(G)=\{v_1 ,\ldots, v_n \}$ and 
		\[
		q= \min\left\{k \; : \;    \Phi_k (H,v)\geq\frac{D(G)}{k}\right\}.
		\]
		Obviously, $q\geq D(H,v)$ and therefore $\Phi_q (H,v)>0$. Let $\mathcal{A}_q$ be the set of  non-equivalent distinguishing colorings of $(H,v)$ with at most $q$ colors and set
		\[
		\mathcal{B}_q=\left\{ (a,f)\; : \; 1\leq a\leq q, f\in\mathcal{A}_q \right\}.
		\]
		We have $\vert \mathcal{A}_q \vert = \Phi_q (H,v)$ and  $\vert \mathcal{B}_q \vert= q\cdot \Phi_q (H,v) \geq D(G)$. 
		Hence, there is a distinguishing coloring~$c$ of $G$ using elements of $\mathcal{B}_q$ as colors. Define a $q$-coloring $\tilde{c}$ of $G_s(H,v)$ as follows:
		\[
		\tilde{c}(x)=\left\{ \begin{array}{lll}
			a; && \textnormal{ if  }\; x=v_i,\; i\in \{1,\ldots,n\}\textnormal{ and } c(x) = (a,f),\\[5pt]
			f(x); &&  \textnormal{ if  } \;  x\in H_i -  \{v_i\},\; i\in \{1,\ldots,n\} \textnormal{ and } c(v_i)=(a,f).
		\end{array}\right.
		\]
		One can easily see that $\tilde{c}$ is a distinguishing coloring, which means that $D(G_s(H,v))\leq q$. 
		
		On the other hand, if $r<q$, then $\vert \mathcal{B}_r \vert =r\cdot \Phi_r (H,v)<D(G)$ which means that any coloring of $G$ with elements of $\mathcal{B}_r$ cannot be distinguishing. Therefore, there is no $r$-coloring which breaks both symmetries in copies of $(H,v)$ and symmetries in $G$. Consequently, $D(G_s(H,v))\geq q$.
	\end{proof}

	\begin{theorem}
		Let $G$ and $H$ be connected graphs, $H$ be rooted at $v$ and $\vert G \vert ,\vert H \vert\geq 2$.
\begin{itemize}
\item[a.]
 If both $G$ and $H$ are asymmetric graphs, then $\theta\left(G_s (H)\right)=1$.

\item[b.]
If both $G$ and $H$ are not asymmetric or $G$ is an asymmetric and $H$ is not, then
	\[
		\theta \left(G_s(H,v)\right)=\left(\vert G\vert-1\right)\cdot \vert H\vert+\theta (H,v),
		\]
		
\item[c.]
If $\aut(G) \ne \{id\}$ and $H$ is an asymmetric graph, then
\[
		\theta \left(G_s(H,v)\right)=\left(\theta(G)-1\right)\cdot \vert H\vert+1.
		\]

\end{itemize}
	\end{theorem}
	\begin{proof}
		The first part is clear by Theorem \ref{aut-rooted} and the fact that the distinguishing threshold of an asymmetric graph is~1. Hence, we may assume that $H$ has some symmetries and  $t=\left(\vert G\vert-1\right)\cdot \vert H\vert+\theta (H,v)-1$. We color the  vertices of $G_s(H,v)$ with $t$ colors in such a way  that $(H_1,v_1) -  \{v_1\}$ receives a non-distinguishing $\left(\theta(H,v)-1\right)$-coloring and all other vertices receive the remaining colors. The resulting coloring is not distinguishing because $\id\circ(\gamma,\id,\ldots,\id)$ preserves it, where $\gamma$ is a non-identity automorphism of $(H,v)$ which preserves the $\left(\theta(H,v)-1\right)$-coloring of $(H_1,v_1)$. Therefore, $\theta \left(G_s(H,v)\right)>t$.
		
		Conversely, let $c$ be a  $(t+1)$-coloring of $G_s(H,v)$. According to Theorem~\ref{aut-rooted}, any automorphism $\alpha\in\aut(G_s(H,v))$ is of the form $\alpha=\beta\circ (\gamma_1,\ldots,\gamma_n)$, for some $\beta \in \Aut(G)$ and $\gamma_1,\ldots,\gamma_n\in\Aut(H,v)$.  Since $c$ assigns  $(H_i,v_i)$ a color which is not assigned to $(H_j,v_j)$, for any $i\neq j$, we have that $\beta=\id$. On the other hand, $(H_i, v_i)$ receives no fewer color than  $\theta(H,v)$, which implies that $\gamma_i=\id$, for each $i$; that is, $\alpha=\id\circ (\id,\ldots,\id)=\id$. This shows that $c$ is distinguishing and the result follows.

In order to prove part (c), According to equation~(\ref{max-lem}), it suffices to show that \[ \left(\theta(G)-1\right)\cdot \vert H\vert+1=\max\left\{c(\tilde{\alpha})\;:\; \tilde{\alpha}\in \aut(G_s(H,v)) \right\}+1 .\]
It is easy to see that $\aut(G) \cong \aut(G_s(H,v))$ and every automorphism of $\aut(G_s(H,v))$ is an extension of an automorphism $\alpha \in \aut(G)$. Let $\alpha \ne id$ be an automorphism of $G$ such that $c(\alpha) = \max\left\{c(\sigma)\;:\; \sigma\in \aut(G) \right\}$ which results in 
$$c(\tilde{\alpha}) = \max\left\{c(\tilde{\beta})\;:\; \tilde{\beta}\in \aut(G_s(H,v)) \right\}.$$
Suppose that $\alpha = \sigma_1 \sigma_2 \cdots \sigma_r \gamma_1 \gamma_2 \cdots \gamma_s$ be the cycle decomposition of $\alpha$. Using properties of $\tilde{\alpha}$ one can deduce that for every cycle in $\alpha$, there are $|H|$ cycles in $\tilde{\alpha}$ and for every fixed vertex of $\alpha$ there are $|H|$ fixed vertices of $\tilde{\alpha}$. Hence, \[\theta(G_s(H,v)) = c(\tilde{\alpha}) +1 = |H|\cdot c(\alpha) + 1= |H|\cdot (\theta(G) - 1) +1\]
and the proof is complete.
	\end{proof}

	\section{Corona product}\label{corona}
	
	In this section we study the distinguishing number and the distinguishing threshold of the  corona product  of two graphs. We recall that given two graphs $G$  and  $H$, their corona product $G\odot H$ is a graph on $\vert G\vert\cdot \left(\vert H \vert+1\right)$ vertices obtained by taking one copy of $G$ and $\vert G\vert$ copies of $H$ and joining each vertex of the $i$-th copy of $H$ to the $i$-th vertex of $G$, for $1\leq i \leq \vert G\vert$. An illustration for this product is drawn in Figure~\ref{fig:corona}. The reader may refer to~\cite{HIK2011} for more details.
	
	\begin{figure}[h]
		\centering
		\begin{center}
			\begin{tabular}{ c c c c  c c c c }
				\begin{tikzpicture}  [scale=0.6]
					\tikzstyle{every path}=[line width=1pt]
					
					\newdimen\ms
					\ms=0.1cm
					\tikzstyle{s1}=[color=black,fill,rectangle,inner sep=3]
					\tikzstyle{c1}=[color=black,fill,circle,inner sep={\ms/8},minimum size=2*\ms]
					
					
					\coordinate (a1) at  (0,0);
					\coordinate (a2) at  (2,0);
					\coordinate (a3) at (4,1);
					\coordinate (a4) at (6,0);

					\draw [color=black] (a1) -- (a2);
					\draw [color=black] (a2) -- (a3);
					\draw [color=black] (a2) -- (a4);
					\draw [color=black] (a3) -- (a4);
					\draw (a1) coordinate[c1];
					\draw (a2) coordinate[c1];
					\draw (a3) coordinate[c1];
					\draw (a4) coordinate[c1];
					
				\end{tikzpicture}
				& &
				\begin{tikzpicture}  [scale=0.4]
					
					\tikzstyle{every path}=[line width=1pt]
					
					\newdimen\ms
					\ms=0.1cm
					\tikzstyle{s1}=[color=black,fill,rectangle,inner sep=3]
					\tikzstyle{c1}=[color=black,fill,circle,inner sep={\ms/8},minimum size=2*\ms]
					

					\coordinate (a6) at (0,0);
					\coordinate(a7) at (0.9,2);
					\coordinate(a8) at (2,0);
					\coordinate(a9) at (1.1,-2);
					

					\draw [color=black] (a6) -- (a7);
					\draw [color=black] (a7) -- (a8);
					\draw [color=black] (a8) -- (a9);
					\draw [color=black] (a6) -- (a9);

					\draw (a6) coordinate[c1];
					\draw(a7) coordinate[c1];
					\draw(a8) coordinate[c1];
					\draw(a9) coordinate[c1];
				\end{tikzpicture}
				& & &&
				\begin{tikzpicture}  [scale=0.8]
					\tikzstyle{every path}=[line width=1pt]
					
					\newdimen\ms
					\ms=0.1cm
					\tikzstyle{s1}=[color=black,fill,rectangle,inner sep=3]
					\tikzstyle{c1}=[color=black,fill,circle,inner sep={\ms/8},minimum size=1.5*\ms]
					
					
					\coordinate (a1) at  (0,0);
					\coordinate (a2) at  (2,0);
					\coordinate (a3) at (4,1);
					\coordinate (a4) at (6,-0);
					\coordinate (a5) at  (0.1,1);
					\coordinate (a6) at  (0.5,2);
					\coordinate (a7) at  (-0.1,3);
					\coordinate (a8) at  (-0.5,2);
					\coordinate (a9) at  (2.1,1);
					\coordinate (a10) at (2.5,2);
					\coordinate (a11) at  (1.9,3);
					\coordinate (a12) at  (1.5,2);
					\coordinate (a13) at  (4.1,2);
					\coordinate (a14) at  (4.5,3);
					\coordinate (a15) at  (3.9,4);
					\coordinate (a16) at (3.5,3);
					\coordinate (a17) at  (6.1,1);
					\coordinate (a18) at  (6.5,2);
					\coordinate (a19) at (5.9,3);
					\coordinate (a20) at (5.5,2);

					\draw [color=black] (a1) -- (a2);
					\draw [color=black] (a2) -- (a3);
					\draw [color=black] (a2) -- (a4);
					\draw [color=black] (a3) -- (a4);
					\draw [color=black] (a1) -- (a5);
					\draw [color=black] (a1) -- (a6);
					\draw [color=black] (a1) -- (a7);
					\draw [color=black] (a1) -- (a8);
					\draw [color=black] (a5) -- (a6);
					\draw [color=black] (a6) -- (a7);
					\draw [color=black] (a7) -- (a8);
					\draw [color=black] (a5) -- (a8);
					\draw [color=black] (a2) -- (a9);
					\draw [color=black] (a2) -- (a10);
					\draw [color=black] (a2) -- (a11);
					\draw [color=black] (a2) -- (a12);
					\draw [color=black] (a9) -- (a10);
					\draw [color=black] (a10) -- (a11);
					\draw [color=black] (a11) -- (a12);
					\draw [color=black] (a9) -- (a12);
					\draw [color=black] (a3) -- (a13);
					\draw [color=black] (a3) -- (a14);
					\draw [color=black] (a3) -- (a15);
					\draw [color=black] (a3) -- (a16);
					\draw [color=black] (a13) -- (a14);
					\draw [color=black] (a14) -- (a15);
					\draw [color=black] (a15) -- (a16);
					\draw [color=black] (a13) -- (a16);
					\draw [color=black] (a4) -- (a17);
					\draw [color=black] (a4) -- (a18);
					\draw [color=black] (a4) -- (a19);
					\draw [color=black] (a4) -- (a20);
					\draw [color=black] (a17) -- (a18);
					\draw [color=black] (a18) -- (a19);
					\draw [color=black] (a19) -- (a20);
					\draw [color=black] (a17) -- (a20);					
					\draw (a1) coordinate[c1];
					\draw (a2) coordinate[c1];
					\draw (a3) coordinate[c1];
					\draw (a4) coordinate[c1];
					\draw (a5) coordinate[c1];
					\draw (a6) coordinate[c1];
					\draw (a7) coordinate[c1];
					\draw (a8) coordinate[c1];
					\draw (a9) coordinate[c1];
					\draw (a10) coordinate[c1];
					\draw (a11) coordinate[c1];
					\draw (a12) coordinate[c1];
					\draw (a13) coordinate[c1];
					\draw (a14) coordinate[c1];
					\draw (a15) coordinate[c1];
					\draw (a16) coordinate[c1];
					\draw (a17) coordinate[c1];
					\draw (a18) coordinate[c1];
					\draw (a19) coordinate[c1];
					\draw (a20) coordinate[c1];
				\end{tikzpicture}
				\\ \vspace{1mm} \\
				$G$ & &$H$& &  && $G\odot H$\\
			\end{tabular}

		\end{center}
		\caption{the corona product of two graphs}
		\label{fig:corona}
	\end{figure}
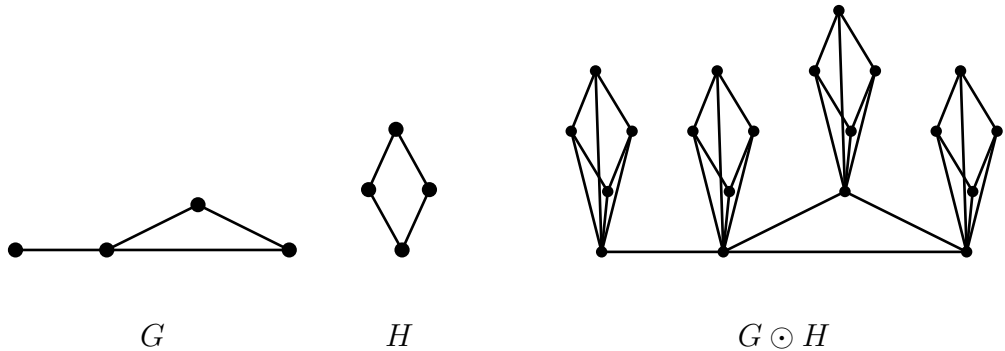

	The automorphisms of the corona product graphs were studied in~\cite{Alikhani2017corona} with an erratum. The correct version is as follows: when $G\not\cong K_1$, any $\alpha\in\aut(G\odot H)$ can be expressed by a combination of an automorphism $\beta\in\aut(G)$ and $\vert G\vert$ automorphisms $\gamma_1 ,\ldots \gamma_{\vert G\vert}\in\aut(H)$, i.e., $\alpha=\beta\circ\gamma_1\circ\ldots\circ\gamma_{\vert G\vert}$. Consequently, we have 
\begin{equation}
	\vert \aut(G\odot H)\vert=\vert\aut(G)\vert\cdot\vert\aut(H)\vert^{\vert G\vert}.
\end{equation}

	It was also shown that when $D(G)\leq D(H)$, the distinguishing number of the corona product $G\odot H$ is $D(H)$. Moreover, for the case where $D(G)>D(H)$, some lower and upper bounds  were also given~\cite{Alikhani2017corona}. In this section,   we obtain similar results as in Section~\ref{rooted}. We first evaluate $D(G\odot H)$ precisely in terms of $\Phi_k (H)$.
	
	\begin{theorem}
		Let $G$ and $H$ be two graphs and $G\not\cong K_1$. Then 
		\[
		D(G\odot H) = \min\left\{k\; : \; \Phi_k (H)\geq\frac{D(G)}{k}\right\}.
		\]
	\end{theorem}
	
	\begin{proof}
	When $D(G)\leq D(H)$ we have $D(G\odot H)=D(H)$ \cite{Alikhani2017corona}. In this case, the minimum $k$ such that $k\cdot \Phi_k (H)\geq D(G)$ is $D(H)$ because $\Phi_k (H)=0$ for $k<D(H)$. Consequently, the statement follows.
		
		Suppose that $D(G)>D(H)$ and $q= \min\left\{k\; :  \; \Phi_k (H)\geq\frac{D(G)}{k}\right\}$. In this case, the argument is quite similar to the proof of Theorem \ref{dist-rooted}, so we only need to rephrase it. Obviously, $q\geq D(H)$ and therefore $\Phi_q (H)>0$. Let $\mathcal{A}_q$ be the set of  non-equivalent distinguishing colorings of $H$ with at most $q$ colors and set 
		\[
		\mathcal{B}_q=\left\{ (a,f)\; : \; 1\leq a\leq q, f\in\mathcal{A}_q \right\}.
		\]
		Therefore  $\vert \mathcal{A}_q \vert = \Phi_q (H)$  and  $\vert \mathcal{B}_q \vert= q\cdot \Phi_q (H) \geq D(G)$. 
		Hence, there is a distinguishing coloring $c$ of $G$ using elements of $\mathcal{B}_q$ as colors. Define a $q$-coloring $\tilde{c}$ of $G\odot H$ as follows:
		\[
		\tilde{c}(v)=\left\{ \begin{array}{lll}
			a; &&\textnormal{ if  } v\in G\textnormal{ and } c (v) = (a,f),\\[5pt]
			f(v); &&\textnormal{ if  } v\in H_u , u\in G \textnormal{ and } c(u)=(a,f).
		\end{array}\right.
		\]
		We can easily see that $\tilde{c}$ is a distinguishing coloring, which means that $D(G\odot H)\leq q$. 
		
		Finally, if $r<q$, then $\vert \mathcal{B}_r \vert =r\cdot \Phi_r (H)<D(G)$ which means that any coloring of $G$ with elements of $\mathcal{B}$ cannot be distinguishing. Therefore, there is no $r$-coloring that breaks both symmetries in copies of $H$ and symmetries in $G$. Consequently, $D(G\odot H)\geq q$.
	\end{proof}

	We point out that Alikhani and Soltani in~\cite{Alikhani2017corona} showed that $D(H)\leq D(K_1 \odot H)\leq D(H)+1$, which completes the study.
		We conclude this section by evaluating   the distinguishing threshold of the corona product of two graphs.

	\begin{theorem}\label{corona_theta}
		Let $G$ and $H$ be two graphs.
		
		\begin{itemize}
			\item[(a)] If $\aut(H)\neq\{\id\}$, then $\theta(G\odot H)=|G| + |H|\cdot(|G|-1) + \theta(H)$.
			
			\item[(b)] If $ \aut (H)=\{\id\}$, then $\theta(G\odot H)=(|H|+1) \cdot\theta(G)- |H|$.
		\end{itemize}
		
	\end{theorem}
	\begin{proof}
		Let $|G|=n$ and $|H|=m$. To prove  (a), assume $s=n + m(n-1) + \theta(H)$.  It is evident that $\theta(G\odot H)\geq s$, thus it suffices to show that any coloring of $G\odot H$ with $s$ colors is distinguishing. Any such coloring  assigns unique colors to  $s$ vertices of $G \odot H$ and we assume, without loss of generality, that a remaining  subset $X\subset V(G\odot H)$, with  $|X|=m-\theta(H)$, receives only one color, say~$1$.  If $X\cap V(G)= \emptyset$, then in any copy of $H$ in $G\odot H$, at least $\theta(H)$ colors have been used that are not used elsewhere. This shows that the coloring of each copy of $H$ and, hence, the coloring of the entire graph $G \odot H$ is distinguishing. If, otherwise, $X\cap V(G)\neq \emptyset$, then any copy of $H$ contains a color which no other copy does. Note  that if $\alpha\in \aut(G\odot H)$ maps $v_1\in G$ to  $v_2\in G$, then it has to map the copy of $H$ adjacent to $v_1$ onto the copy of $H$ adjacent to $v_2$. Therefore, no automorphism of $G\odot H$ can preserve the coloring. 
		
		To prove (b),  we note that there is a group  isomorphism $f:\aut(G)\to \aut(G\odot H)$.  Assume that $\alpha$ is a non-identity automorphism of $G$ with maximum number of disjoint cycles in its cycle decomposition. Then $\tilde{\alpha}=f(\alpha)$ is a non-identity automorphism of $G\odot H$ with maximum number of disjoint cycles in its cycle decomposition and $c(\tilde{\alpha})=(m+1) c(\alpha)$. On the other hand, according to equation~(\ref{max-lem}), $c(\alpha)=\theta(G)-1$; hence
		$c(\tilde{\alpha})= (m+1) (\theta(G)-1)$.
		Using equation~(\ref{max-lem}) again, the result follows.
	\end{proof}

	\section{Lexicographic product}\label{lexico}
	
	Finally, we consider the  lexicographic product of graphs. Given two graphs $G$ and $H$, their lexicographic product is the graph  $Z=G\circ H$  with 
	\[
	V(Z)= \{ (x,y)\; :\; x \in G,\, y\in H\},
	\]
	in which $(x,y)$ is adjacent with $(x',y')$ if  $x\sim x'$ in $G$  or  $x = x'$ and  $y\sim y'$ in $Y$. It is easy to see that for every vertex $x\in G$, there is a copy of $Y$, say $Y_x$ in $C\circ H$.
	A simple illustration of this product is depicted in Figure~\ref{fig:lexico}.
	
	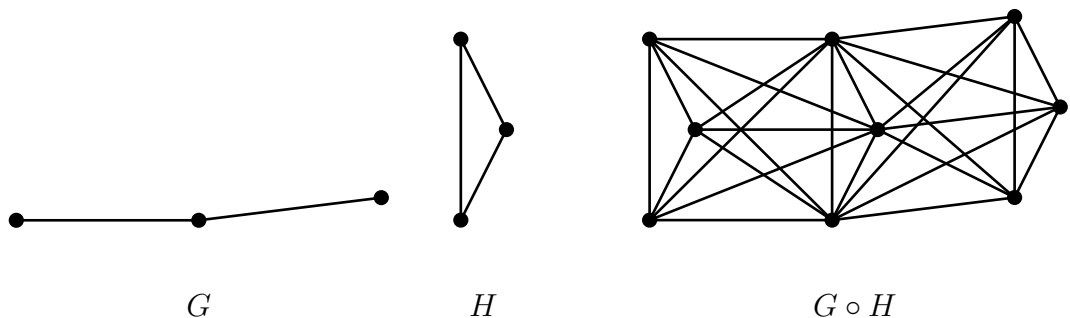
\begin{figure}[h]
		\centering
		\begin{center}
			\begin{tabular}{ c c c c  c c c c }
				\begin{tikzpicture}  [scale=0.6]
					\tikzstyle{every path}=[line width=1pt]
					
					\newdimen\ms
					\ms=0.1cm
					\tikzstyle{s1}=[color=black,fill,rectangle,inner sep=3]
					\tikzstyle{c1}=[color=black,fill,circle,inner sep={\ms/8},minimum size=2*\ms]
					
					
					\coordinate (a1) at  (0,0);
					\coordinate (a2) at  (4,0);
					\coordinate (a3) at (8,0.5);

					\draw [color=black] (a1) -- (a2);
					\draw [color=black] (a2) -- (a3);
			
					\draw (a1) coordinate[c1];
					\draw (a2) coordinate[c1];
					\draw (a3) coordinate[c1];

				\end{tikzpicture}
				& &
				\begin{tikzpicture}  [scale=0.6]
					
					\tikzstyle{every path}=[line width=1pt]
					
					\newdimen\ms
					\ms=0.1cm
					\tikzstyle{s1}=[color=black,fill,rectangle,inner sep=3]
					\tikzstyle{c1}=[color=black,fill,circle,inner sep={\ms/8},minimum size=2*\ms]
					

					\coordinate (a6) at (0,2);
					\coordinate(a8) at (0,-2);
					\coordinate(a9) at (1,0);
					

					\draw [color=black] (a6) -- (a8);
					\draw [color=black] (a8) -- (a9);
					\draw [color=black] (a6) -- (a9);

					\draw (a6) coordinate[c1];
					\draw(a8) coordinate[c1];
					\draw(a9) coordinate[c1];
				\end{tikzpicture}
				& & &&
				\begin{tikzpicture}  [scale=0.6]
					\tikzstyle{every path}=[line width=1pt]
					
					\newdimen\ms
					\ms=0.1cm
					\tikzstyle{s1}=[color=black,fill,rectangle,inner sep=3]
					\tikzstyle{c1}=[color=black,fill,circle,inner sep={\ms/8},minimum size=2*\ms]
					
					
					\coordinate (a1) at  (-1,0);
					\coordinate (a2) at  (3,0);
					\coordinate (a3) at (7,0.5);
					\coordinate (a5) at  (-2,2);
					\coordinate (a7) at (-2,-2);
					\coordinate (a8) at  (2,2);
					\coordinate (a10) at (2,-2);
					\coordinate (a11) at  (6,2.5);
					\coordinate (a13) at (6,-1.5);

					\draw [color=black] (a1) -- (a2);
					\draw [color=black] (a2) -- (a3);
					\draw [color=black] (a1) -- (a5);
					\draw [color=black] (a5) -- (a7);
					\draw [color=black] (a1) -- (a7);
					\draw [color=black] (a2) -- (a8);
					\draw [color=black] (a8) -- (a10);
					\draw [color=black] (a2) -- (a10);
					\draw [color=black] (a3) -- (a11);
					\draw [color=black] (a11) -- (a13);
					\draw [color=black] (a3) -- (a13);
					\draw [color=black] (a5) -- (a8);
					\draw [color=black] (a5) -- (a10);
					\draw [color=black] (a1) -- (a8);
					\draw [color=black] (a1) -- (a10);
					\draw [color=black] (a7) -- (a10);
					\draw [color=black] (a7) -- (a8);
					\draw [color=black] (a7) -- (a2);
					\draw [color=black] (a5) -- (a2);
					\draw [color=black] (a8) -- (a11);
					\draw [color=black] (a8) -- (a13);
					\draw [color=black] (a8) -- (a3);
					\draw [color=black] (a10) -- (a3);
					\draw [color=black] (a10) -- (a11);
					\draw [color=black] (a10) -- (a13);
					\draw [color=black] (a13) -- (a2);
					\draw [color=black] (a11) -- (a2);

					\draw (a1) coordinate[c1];
					\draw (a2) coordinate[c1];
					\draw (a3) coordinate[c1];
					\draw (a5) coordinate[c1];
					\draw (a7) coordinate[c1];
					\draw (a8) coordinate[c1];
					\draw (a10) coordinate[c1];
					\draw (a11) coordinate[c1];
					\draw (a13) coordinate[c1];
				\end{tikzpicture}
				\\ \vspace{1mm} \\
				$G$ & &$H$& &  && $G\circ H$\\
			\end{tabular}

		\end{center}
		\caption{the lexicographic product of two graphs}
		\label{fig:lexico}
	\end{figure}

	Automorphism groups of the lexicographic product graphs were studied by Hemminger \cite{heminger1968} and Sabidussi \cite{Sabidussi1961}.  Hemminger defined natural isomorphisms as follows: let $Z_1 = G_1 \circ H_1$ and $Z_2 =G_2 \circ H_2$. Then a graph isomorphism $\mu: Z_1\rightarrow Z_2$ is called \emph{natural} if for each $x_1 \in G_1$ there is an $x_2 \in G_2$ such that $\mu({H_1}_{x_1})= {H_2}_{x_2}$; otherwise $\mu$ is called \textit{unnatural}. Moreover, he characterized all lexicographic product graphs whose automorphism groups consist  of all their natural automorphisms \cite{heminger1968}. It is clearly  more difficult  to deal with the  unnatural automorphisms, hence here we only consider the lexicographic products whose automorphisms consist only of natural ones.
	In this case, the distinguishing number of the lexicographic product of $G$ and $H$ has already been calculated in~\cite{ahmadi2020number} as the smallest integer $k$ such that $\Phi_k (H)\geq D(G)$. In the following theorem, which is similar to Theorem~\ref{corona_theta}, we calculate the distinguishing threshold of the lexicographic product of such $G$ and $H$.

	\begin{theorem}
		Let $G$ and $H$ be two graphs such that $\aut(G\circ H)$ consists only of natural automorphisms. 
		\begin{itemize}
			\item[(a)] If $\aut(H)\neq\{\id\}$, then $\theta(G\circ H) = (\vert G\vert - 1)\cdot \vert H\vert + \theta(H)$.
			
			\item[(b)] If $ \aut (H)=\{\id\}$, then $\theta\left(G\circ H \right) = \left(\theta(G) -1\right)\cdot \vert H\vert +1$.
		\end{itemize}
		
	\end{theorem}
	\begin{proof}
		To prove (a), 		let $s= (\vert G\vert - 1)\cdot \vert H\vert + \theta(H)$. It is obvious that  $\theta(G\circ H) \geq s$.
		Assume that $c$ is an arbitrary coloring for $G\circ H$ with $s$ colors and that $\alpha\in\aut (G\circ H)$ is arbitrary. We show that $\alpha$ cannot preserve  $c$ and this proves the claim. Since $\alpha$ is natural, we have the following two cases.
		If $\alpha$ maps at least one copy of $H$ to another, then, since the number of colors are more than $(|G| - 1)\cdot |H|$, we deduce that every copy of $H$ has a color that is not used in the other copies of $H$. Therefore, $\alpha$ cannot preserve $c$. If, on the other hand, the restriction of $\alpha$ to each copy of $H$ is an automorphism of that copy of $H$, then the fact that  $s\geq |G| \cdot \theta (H)$ implies that $c$ uses at least $\theta(H)$ colors on each copy of $H$, which  again  shows that $\alpha$ cannot preserve $c$.
		
		To prove (b), we note that, similar to the proof of Theorem~\ref{corona_theta}, there is a group  isomorphism $f:\aut(G)\to \aut(G\circ H)$. For any $\alpha\in \aut(G)$, we denote $f(\alpha)$ by $\tilde{\alpha}$.  According to equation~(\ref{max-lem}), there exists an automorphism $\lambda$ such that $\theta(G) = c( \lambda)+1$. Since $\aut(H) = \{\id\}$ and all the automorphisms of $G\circ H$ are natural, we have
		\[\max\left\{c( \tilde{\alpha}) \;: \; \tilde{\alpha} \in \aut(G\circ H)\right\} = c( \tilde{\lambda}). \]
		On the other hand, if $\sigma_j$ is a cycle in the cycle decomposition of $\lambda$, then $\tilde{\sigma}_j$ contains $|H|$ disjoint cycles of length $o(\sigma_j)$. This implies that $c(\tilde{\lambda}) = |H| \cdot c(\lambda) = |H|\cdot (\theta(G) - 1)$ 
		which, using equation~(\ref{max-lem}), completes  the proof.
	\end{proof}

	\section{Conclusion and future work}\label{conclusion}
	
	Following the recent works concerning  the distinguishing number and the distinguishing  threshold of graphs,   we addressed the problem of determining these parameters for some graph operations, namely vertex-sums, rooted products, corona products and lexicographic products. 
	The calculations of distinguishing number   consisted of some combinatorial arguments  while for computing the distinguishing threshold, we analyzed the automorphisms for their cycle structures.
	
	To handle the first operation, we introduced the concept of steady vertices which turns out to have some nice properties on its own and seems to have some interesting applications as well. As one of future research directions in this regard, one can extend the notion of vertex-sum to $m$-vertex-sums where new graphs are obtained by identifying an $m$-set of vertices in each graph. Then the new concept  of a steady vertex generalizes to  a  \emph{steady set}. As an auxiliary result, we showed that Theorem~\ref{steady_iff} can be generalized to steady sets, as well. However 
   generalizing Theorem~\ref{dvsum} seems to be a harder problem.  Since every connected graph is an $m$-vertex-sum of some smaller graphs, studying  this problem may be a fruitful approach  to obtain upper bounds for the distinguishing number and the  distinguishing threshold of arbitrary graphs.

	As another generalization, we can analogously obtain some results about \emph{edge-sums}, where graphs are glued to each other by identifying one of their edges. As an example of such calculations, we observe that    the distinguishing number of the edge-sum of $t$ cycles  $C_n$ is equal to $ \left\lceil{\sqrt[n-2]{t}}\right\rceil$. 
	
	The ultimate outcome of the present research is that calculating $\varphi_k (G)$ and $\Phi_k (G)$ is more valuable than we previously thought. We showed that these indices are important in calculating distinguishing indices of the graph operations we considered here and they might appear in some other types of operations. Since knowing the distinguishing threshold of a graph $G$ is an important factor in calculating $\varphi_k (G)$ and $\Phi_k (G)$, an interesting plan for a future study is calculating these indices for much more classes of graphs that they are known now.

	
	\bibliographystyle{plain}
	\bibliography{bibliography}
	
\end{document}